\documentclass{article}
\usepackage{amsmath,amssymb}
\usepackage{graphicx}
\usepackage{hyperref}
\allowdisplaybreaks[1]
\usepackage[margin=2.25cm]{geometry}

\newtheorem{thm}{Theorem}[section]

\newtheorem{lemma}[thm]{Lemma}

\newtheorem{theorem}[thm]{Theorem}
\newtheorem{remark}[thm]{Remark}

\newtheorem{conj}[thm]{Conjecture}

\newenvironment{proof}{{\bf Proof:}}{\hfill$\square$\vskip.5cm}

\newcommand{\R}{\mathbb{R}}
\newcommand{\N}{\mathbb{N}}
\newcommand{\C}{\mathbb{C}}

\newcommand{\Z}{\mathbb{Z}}

\begin{document}


\title{Pal's permanent conjecture: proof for block-uniform matrices}

\author{Andrea Ottolini\footnote{\href{mailto:ottolini@uab.edu}{ottolini@uab.edu}} and Shannon Starr\footnote{\href{mailto:slstarr@uab.edu}{slstarr@uab.edu}}\\
University of Alabama at Birmingham,\\
1402 Tenth Avenue South, Birmingham, AL 35294--1241, USA}

\maketitle

\begin{abstract}
Consider a function $\mathcal{C}(x,y)$ on $[0,1]\times[0,1]$, and let $A^{(n)} = \big(a^{(n)}_{i,j}\, :\, i,j \in [n]\big)$ be the matrix with entries $a^{(n)}_{i,j}\, =\, \exp(-\mathcal{C}(i/n,j/n))$.
For $\mathcal{C}$ that is smooth enough and symmetric, Soumik Pal conjectured the asymptotics \[\operatorname{perm}\big(A^{(n)}\big)/n!\sim \exp\big(n \Lambda[\mathcal{C}]\big)/ \sqrt{\mathcal{D}[\mathcal{C}]}\] as $n \to \infty$
for known functionals that arise naturally in the context of entropy regularised optimal transport. The functional  $\Lambda[\mathcal{C}]$ is the known large deviation rate function, already proved rigorously by Sumit Mukherjee.
It is 
$\int_{0}^1 \int_0^{1} (\alpha(x)+\beta(y))\, dx\, dy$ where $\alpha(x)+\beta(y)$ is chosen such that 
$\rho(x,y) := \exp(-\mathcal{C}(x,y)-\alpha(x)-\beta(y))$ has uniform marginals.
The algebraic term $\mathcal{D}[\mathcal C]$ is given by Peter McCullagh's formula for doubly stochastic matrices:
$\operatorname{det}_F(I+J-T^2)$, the Fredholm determinant, where $I$ is the identity
on $L^2([0,1])$, $Jf(x) \equiv \int_{0}^1 f(z)\, dz$ and $Tf(x) = \int_0^1 \rho(x,y) f(y)\, dy$. We show that smoothness and symmetry are not essential for the conjecture to hold by proving that a similar result holds for functions $\mathcal C$ that are constant on blocks, exploiting a well-known Ross Pinsky's combinatorial decomposition of permutations in blocks.

\bigskip

{\small

{\bf keyword:}
permanents; block-uniform matrices; optimal transport;
Schr\"odinger bridges; random permutations; local central limit theorem
\medskip

{\bf msc numbers:}
primary 60C05; Secondary 15A15.
}
\end{abstract}

\section{Introduction}

Suppose that $A = \big(a_{i,j}\, :\, i,j \in [n]\big)$ is a matrix
where $[n]=\{1,\dots,n\}$.
The permanent is 
\[
\operatorname{perm}(A)\, =\, \sum_{\pi \in S_n} \prod_{i=1}^n a_{i,\pi_i}\, ,
\]
which differs from the determinant in the signs of some of the summands
\[
\det(A)\, =\, \sum_{\pi \in S_n} (-1)^{\sigma(\pi)} \prod_{i=1}^{n} a_{i,\pi_i}\, ,
\]
where $\sigma(\pi)=0$ for even permutations  and $\sigma(\pi)=1$
for odd permutations.
Symmetry properties of the determinant make it easy to compute and approximate.
But for the permanent, typical calculations can be quite difficult.

For example, Valiant proved that the computational complexity of 
calculating the permanent of an arbitrary $n \times n$ zero-one matrix, a matrix
such that every $a_{i,j}$ is either $0$ or $1$,
is \#P-complete \cite{Valiant}.
In the same vein, a more recent conjecture of Aaronson and Arkhipov
is known as the ``permanent of Gaussians conjecture.''
Suppose that $A$ is a Gaussian random matrix, whose entries are IID
$\mathcal{N}(0,1)$ plus $i$ times independent $\mathcal{N}(0,1)$ random variables
(so the Ginibre ensemble). 
Then the problem of determining for each $\epsilon>0$
and $\delta>0$ an approximation of $\operatorname{perm}(A)$ that is within $\epsilon$ relative error, with probability at least $1-\delta$, is $\#$P-hard.
The reason this conjecture is important is that it provides a key step
for Aaronson and Arkhipov's quantum supremacy theoretical result based on
Boson sampling \cite{AaronsonArkhipov}.

\subsection{Permanents as generalizations of assignments}

More recently, permanents have been considered as a generalization
of the assignment problem from optimal transport theory.
One set of applications of optimal transport theory derives from economics.
See, for instance \cite{Galichon} and references, therein.
In 2012, L\'eonard wrote a seminal article \cite{Leonard}.
Suppose that every $a_{i,j}$ is positive.
The optimal assignment given $A$ would be the permutation $\pi \in S_n$
which maximizes $\prod_{i=1}^{n} a_{i,\pi_i}$ or, up to taking logarithms, the cost function $\sum_{i=1}^{n} \ln(a_{i,\pi_i})$.
A regularized version of this problem is to calculate the permanent
and randomly sample $\pi$ according to the distribution
$\prod_{i=1}^{n} a_{i,\pi_i}/\operatorname{perm}(A)$.
Therefore, interest arose in the context of this problem, especially
given certain conditions on the matrix entries $a_{i,\pi_i}$ beyond
just positivity.

In \cite{Mukherjee}, Mukherjee considered a cost function $\mathcal{C} : [0,1] \times [0,1] \to \R$ satisfying that $\mathcal{C}$ is jointly continuous on its domain and 
\[
\begin{gathered}
\forall x \in [0,1]\, ,\ \text{ we have }\ \int_0^1 \mathcal{C}(x,y)\, dy \, =\, 0\
\text{ and }\\
\forall y \in [0,1]\, ,\ \text{ we have }\ \int_0^1 \mathcal{C}(x,y)\, dx \, =\, 0\, .
\end{gathered}
\]
Then, for each $n \in \N$, defining the matrix 
$A^{(n)} = \big(a^{(n)}_{i,j}\, :\, i,j \in [n]\big)$ given by
\[
a^{(n)}_{i,j}\, =\, \exp\left(- \mathcal{C}(i/n,j/n)\right)\, ,
\]
one has
\[
\lim_{n \to \infty} \frac{1}{n}\, 
\ln\left(\frac{\operatorname{perm}(A^{(n)})}{n!}\right)\,
=\, \Lambda[\mathcal{C}]\, ,
\]
where 
\[
\Lambda[\mathcal{C}]\, =\, \int_0^1 \int_0^1 (\alpha(x)+\beta(y))\, dx\, dy\, ,
\]
for the functions $\alpha,\beta : [0,1] \to \R$ such that $\alpha(x)+\beta(y)$ 
is uniquely
determined by the conditions that
\[
\forall x \in [0,1]\, ,\ \text{ we have }\ \int_0^1 
\exp\left(- \mathcal{C}(x,y) - \alpha(x) - \beta(y)\right)\, dy \, =\, 1\,
\]
and
\[
\forall y \in [0,1]\, ,\ \text{ we have }\ \int_0^1 
\exp\left(- \mathcal{C}(x,y) - \alpha(x) - \beta(y)\right)\,  dx \, =\, 1\, .
\]
In \cite{Pal}, Soumik Pal noted that the determination of the functions $\alpha$ and $\beta$
 is known as the Schr\"odinger bridge problem.
The choice of $\alpha$ and $\beta$ is unique 
up to $\alpha \leftarrow \alpha+c$,
$\beta\leftarrow \beta-c$, because these two shifts cancel in $\alpha(x)+\beta(y)$.

Pal also posed an interesting
conjecture \cite{Pal}.
\begin{conj}
If $\mathcal{C} : [0,1] \times [0,1] \to \R$ is twice continuously differentiable up to the boundary,
symmetric in the two variables, zero on the diagonal and also satisfies the symmetry
$\mathcal{C}(1-x,1-y) = \mathcal{C}(x,y)$ then 
\[
\frac{1}{n!}\, \operatorname{perm}\Big(\big(\exp\left(-\mathcal{C}(i/n,j/n)\right)\, :\, i,j \in [n]\big)\Big)\, \sim\, 
\frac{\exp\left(n\Lambda[\mathcal{C}]\right)}{\sqrt{\mathcal{D}[\mathcal{C}]}}\, ,
\]
where $\mathcal{D}[\mathcal{C}]$ is the Fredholm determinant
\[
\mathcal{D}[\mathcal{C}]\, =\, \det\nolimits_F(I+J-T^2)\, ,
\]
for $I$ the identity operator, 
$J$ the rank-1 operator on $L^2([0,1])$ given by $Jf(x) \equiv \int_0^1 f(y)\, dy$
(the constant), and $T$  the trace-class operator
\[
Tf(x)\, =\, \int_0^1 \rho(x,y) f(y)\, dy\, ,
\]
for $\rho(x,y) = e^{-\mathcal{C}(x,y)-\alpha(x)-\beta(y)}$.
\end{conj}

Our main result, Theorem~\ref{thm:main}, establishes the conjectured asymptotic formula in the complementary case where the cost $\mathcal C$ is piecewise constant on a fixed $m\times m$ grid of equal blocks; symmetry is not required.
In this setting, Pinsky's combinatorial formula \cite{Pinsky1} yields a direct proof by reducing the permanent to a sum over a lattice of block-count matrices of dimension $(m-1)^2$.
The determinant prefactor then arises from the Gaussian integral obtained by approximating this lattice sum near its maximizer.

To place this block-count approach in context, we recall the two arguments that originally led Pal \cite{Pal} to his conjecture.
The first argument relies on a previous result
by Harchaoui, Liu and himself \cite{HLP}.
In that paper there was an extra source of randomness, but they were able to obtain precise asymptotics.
When Pal considered removing the extra source of randomness, heuristically, it led to his conjecture.
One might imagine trying to prove his conjecture that way, using a Tauberian theorem.
(See for instance \cite{Abdesselam} for an interesting paper, on a different topic, using the Wiener-Ikehara
Tauberian theorem among other arguments.)

The second argument, which is more direct, is combinatorial.
In \cite{McCullagh}, Peter McCullagh had already conjectured something similar to the conjecture above,
specifically for doubly stochastic matrices.
But then McCullagh did not need to make
as many assumptions on $\mathcal{C}(x,y)$ such as regularity, given this.
McCullagh gives a similar combinatorial argument as in Pal, but without using the spectral decomposition directly.
In Section 4.3 of his article, McCullagh proves that the quantity of interest is a series
\[
h_n^{(0)}(t)\, =\,
\sum_{k=1}^{n} \mathfrak{x}_{n,k} t^{2k}\, ,
\]
where, for each $k \in \{1,\dots,n\}$, one has
\[
\mathfrak{x}_{n,k}\, =\, \frac{1}{2k}\, 
\operatorname{Tr}\left[\left(\frac{\left(\mathcal{E}^{(n)}\right)^T \mathcal{E}^{(n)}}
{n^{\downarrow 2}}\right)^k\right] + O(1/n)\, ,
\]
for $\mathcal{E}^{(n)} = n(A^{(n)}-J^{(n)})$.
He assumes $\mathcal{E}^{(n)}$ is of moderate deviation, which is a precise version
of the assumption that it is order-1.
He points out that the leading order terms give the Taylor series for the log of the determinant.
So he argues that
\[
h_n^{(0)}(1)\, =\,
\sum_{k=1}^{\infty} \mathfrak{x}_{n,k}\, =\, -\frac{1}{2}\, 
\ln\left(\det\left(I_n - \frac{\left(\mathcal{E}^{(n)}\right)^T \mathcal{E}^{(n)}}
{n^{\downarrow 2}}\right)\right) + O(n^{-1})\, .
\]
In this approach, the passage from the coefficientwise expansion to the determinant formula requires control of the remainder term $O(n^{-1})$.

Following this doubly stochastic approach, Tripathi \cite{Tripathi} proved Pal's conjecture in the smooth setting, using McCullagh's permanent asymptotic as an input and controlling the rescaling of the discretized Schr\"odinger bridge kernel.
The proof of Theorem~\ref{thm:main} does not use McCullagh's asymptotic.

\subsection{Organization of the paper}

The remainder of this article is structured as follows. In Section 2, we formalize our setup for block uniform matrices, state our main result (Theorem 2.1), and illustrate it with concrete examples utilizing Pinsky's combinatorial lemma. In Section 3, we present the rigorous proof of Theorem 2.1 by approximating Pinsky's formula, evaluating the resulting Gaussian integrals, and obtaining the final determinant formula via the Schur complement. We conclude Section 3 by reinterpreting our result probabilistically as a local central limit theorem.

In Appendix A, we formulate a conjecture extending the asymptotic formula to complex block-uniform matrices. In Appendix B, we analyze the evolution of the rate function under Ginibre noise.

\section{Main result and examples}

Consider a fixed, finite $m \in \N = \{1,2,\dots\}$.
Let $B$ be an $m\times m$ matrix $B =\big(b_{r,s}\, :\, r,s\in [m]\big)$
with all entries in $(0,\infty)$.
Consider the vectors $\boldsymbol{v}=(v_1,\dots,v_m) \in \R^m$ and $\boldsymbol{w}=(w_1,\dots,w_m) \in \R^m$
with all components in $(0,\infty)$ such that
\begin{equation}
\label{eq:BridgeV}
\forall r\in[m]\, ,\ \text{ we have }\  \sum_{s=1}^{m} b_{r,s} v_r w_s\, =\, 1\, ,
\end{equation}
and
\begin{equation}
\label{eq:BridgeW}
\forall s\in[m]\, ,\ \text{ we have }\  \sum_{r=1}^{m} b_{r,s} v_r w_s\, =\, 1\, .
\end{equation}
These are unique modulo simultaneous transformations $\boldsymbol{v} \leftarrow c \boldsymbol{v}$ and 
$\boldsymbol{w} \leftarrow c^{-1} \boldsymbol{w}$, for $c>0$, similar to what was stated before. (Here $v_r=e^{-\alpha_r}$
and $w_s=e^{-\beta_s}$ for the finite-dimensional potentials.)
Given this, let $T^{(m)} = \big(t^{(m)}_{r,s}\, :\, r,s \in [m]\big)$
be the matrix where
\[
t^{(m)}_{r,s} = b_{r,s} v_r w_s\, .
\]
For the cost $\mathcal C(x,y)=-\ln b_{r,s}$ on the $(r,s)$th pair of equal intervals,
the density with respect to Lebesgue measure is $\rho(x,y)=m t^{(m)}_{r,s}$.
Thus the continuum rate function is
\[
\Lambda[\mathcal C]=-\ln m-\frac1m\left(\sum_{r=1}^m\ln v_r+\sum_{s=1}^m\ln w_s\right).
\]
And let $J^{(m)} = \big(\mathfrak{j}^{(m)}_{r,s}\, :\, r,s \in [m]\big)$ be such that $\mathfrak{j}^{(m)}_{r,s} \equiv 1/m$
(for all $r$ and $s$).
Finally, let $I^{(m)}$ be the $m\times m$ identity matrix.

Now for each $n \in \N$, let us define the matrix $A^{(m,n)}$  such that
$A^{(m,n)} = \big(a^{(m,n)}_{i,j}\, :\, i,j \in [m\cdot n]\big)$ is given by the formula
\[
a^{(m,n)}_{i,j}\, =\, b_{\lceil i/n\rceil,\lceil j/n\rceil}\, .
\]
(Recall that $\lceil x\rceil = \min(\{k \in \Z\, :\, k\geq x\})$.)
Then we have the following, as our main result.
\begin{theorem}
\label{thm:main}
With the setup as above,
\[
\frac{1}{(m\cdot n)!}\, \operatorname{perm}\left(A^{(m,n)}\right)\,
\sim\, \frac{m^{-mn} \left(\prod_{r=1}^{m} v_r\right)^{-n} \left(\prod_{s=1}^{m} w_s\right)^{-n}}
{\sqrt{\det\left(I^{(m)}+J^{(m)}-
{T}^{(m)}\left({T}^{(m)}\right)^T\right)}}\, ,\
\text{ as $n \to \infty$.}
\]
\end{theorem}
\begin{remark}
We will denote the transpose of $T^{(m)}$ as $\left(T^{(m)}\right)^T$.
\end{remark}
As an example, suppose that we started with $B = J^{(m)}$. Then we could take $\boldsymbol{v} = \boldsymbol{w} = (1,\dots,1)$.
And $A^{(m,n)}$ would be $1/m$ times the all-ones matrix.
So we would have exactly $\operatorname{perm}(A^{(m,n)}) = (mn)!/m^{mn}$, which verifies
the theorem in this particular case because in this case ${T}^{(m)} = J^{(m)}$ as well.
And since $J^{(m)}$ is an orthogonal projection, $J^{(m)}=(J^{(m)})^T$ and $J^{(m)}J^{(m)}=J^{(m)}$.
So $\det\left(I^{(m)}+J^{(m)}-
{T}^{(m)}\left({T}^{(m)}\right)^T\right)$ is just $\det(I^{(m)})$ which is 1.

{\bf Example 2:}
As a second example, suppose $m=2$, and 
\[
B\, =\, \frac{1}{2}\, \begin{bmatrix} 1+\delta & 1-\delta \\ 1-\delta & 1+\delta \end{bmatrix}\, ,
\]
for some $\delta \in [0,1)$. (We will return to consider $\delta=1$, at the end of this section.)
Then we can take $\boldsymbol{v}=\boldsymbol{w}=(1,1)$, again. But now
\[
A^{(2,n)}\, =\, \frac{1}{2}\, \begin{bmatrix} (1+\delta) & \dots & (1+\delta) & (1-\delta) & \dots & (1-\delta) \\
\vdots & \ddots & \vdots & \vdots & \ddots & \vdots\\
(1+\delta) & \dots & (1+\delta) & (1-\delta) & \dots & (1-\delta) \\
(1-\delta) & \dots & (1-\delta) & (1+\delta) & \dots & (1+\delta) \\
\vdots & \ddots & \vdots & \vdots & \ddots & \vdots\\
(1-\delta) & \dots & (1-\delta) & (1+\delta) & \dots & (1+\delta) \end{bmatrix}
\]
So
\[
\operatorname{perm}\left(A^{(2,n)}\right)\,
=\, 2^{-2n} \sum_{\nu=0}^{n} (1+\delta)^{2\nu}(1-\delta)^{2n-2\nu} |\mathfrak{S}_{2n}(\nu)|\, ,
\]
where $\mathfrak{S}_{2n}(\nu)$ is equal to the number of bijections $\pi \in \operatorname{Bij}([2n])$
also satisfying
\[
|\{i \in [n]\, :\, \pi(i) \in [n]\}|\, =\, \nu\, .
\]
But it turns out that $|\mathfrak{S}_{2n}(\nu)|$ and similar quantities have previously been enumerated.

\subsection{A lemma of Ross Pinsky}

\begin{lemma}[Pinsky's lemma]
Suppose that $Q=\big(q_{r,s}\, :\, r,s \in [m]\big)$ is a matrix such that $q_{r,s} \in \{0,\dots,n\}$ for each $r,s \in [m]$
and 
\[
\forall r \in [m]\, ,\ \text{ we have }\ \sum_{s=1}^{m} q_{r,s}\, =\, n\, ,
\]
and
\[
\forall s \in [m]\, ,\ \text{ we have }\ \sum_{r=1}^{m} q_{r,s}\, =\, n\, .
\]
(Let us denote all such matrices as $\mathcal{M}(m,n)$.)
Then, defining $\mathfrak{S}_{mn}(Q)$ to be the set of bijections $\pi \in \operatorname{Bij}([mn])$ 
satisfying, for every $(r,s)\in[m]^2$,
\[
\left|\left\{(i,j)\in[mn]^2:
\begin{array}{l}
\lceil i/n\rceil=r,\ \lceil j/n\rceil=s,\\
\pi(i)=j
\end{array}\right\}\right|=q_{r,s},
\]
we have
\begin{equation}
\label{eq:Pinsky}
|\mathfrak{S}_{mn}(Q)|\, =\, \frac{(n!)^{2m}}{\prod_{r=1}^{m} \prod_{s=1}^{m} q_{r,s}!}\, .
\end{equation}
\end{lemma}
\begin{proof}
For each $(r,s) \in [m]\times[m]$ let $I_{r,s}$ be a $q_{r,s}$-sized subset of $\{i \in [mn]\, :\, \lceil i/n\rceil=r\}$
and let $J_{r,s}$ be a $q_{r,s}$-sized subset of $\{j \in [mn]\, :\, \lceil j/n\rceil=s\}$.
If we demand that $I_{r,1},I_{r,2},\dots,I_{r,m}$ are all disjoint, as we should, then the total number of choices
for that whole row of sets is the multinomial coefficient $n!/(\prod_{s=1}^m q_{r,s}!)$.
Similarly, demanding that $J_{1,s},J_{2,s},\dots,J_{m,s}$ are all disjoint results in the number of choices of these sets
is $n!/(\prod_{r=1}^{m} q_{r,s}!)$.
Of course, we must take the product over $r \in [m]$ and over $s \in [m]$.
But then in each block $\{(i,j)\, :\, \lceil i/n \rceil = r\, ,\ \lceil j/n \rceil=s\}$ we also may match the $q_{r,s}$
points of $I_{r,s}$ with the $q_{r,s}$ points of $J_{r,s}$. This matching corresponds to a bijection in 
$\operatorname{Bij}([q_{r,s}])$. So the number of choices is then also multiplied by $q_{r,s}!$ for each $(r,s) \in [m]\times [m]$.
\end{proof}

This simple result is actually used by Pinsky to great effect in \cite{Pinsky1}.
He shows that for a uniform random permutation $\pi \in S_n$, 
the number $\mathcal{N}(n,k)$ of increasing subsequences of cardinality $k$
has both an explicit formula for the first moment (as known to Hammersley for example)
and an explicit formula for the second moment that Pinsky found, involving one summation.
Then using the Paley-Zygmund second moment method, he found a sufficient condition on the sequence $k=k_n$ for the 
number $\mathcal{N}(n,k_n)$ to satisfy a law of large numbers.

In another beautiful paper \cite{Pinsky2}, he finds a necessary condition for $\mathcal{N}(n,k_n)$
to satisfy a law of large numbers.
Interestingly, he poses a problem which remains open to this day: what is the precise condition on $k_n$
as $n \to \infty$, which is both necessary and sufficient for a law of large numbers to follow for $\mathcal{N}(n,k_n)$?
His two conditions do not match precisely.

(With Samen Hossain, one of us also has an article building on \cite{Pinsky1}.
In theoretical physics, spin glass theorists interested in calculating a physical quantity
like the ``quenched free energy'' would ask for the precise asymptotics of all the integer moments.
This would be a first step for them towards making a replica-symmetry-breaking ansatz.
Motivated by this, we wrote an article considering the precise asymptotics for the 2nd moment in the case of $k_n=cn^{1/2}$ \cite{HossainStarr}.)

In \cite{Mukherjee}, Sumit Mukherjee also rederived this result, and used it to great effect in deriving the large deviation
rate function for the permanent problem.

\subsection{Completion of Example 2}
We will complete the proof of Theorem \ref{thm:main} in the next section.
For now, let us return to our second example.
With the notation as before, we can write $|\mathfrak{S}_{2n}(\nu)|$
as  $\left|\mathfrak{S}_{2n}\left(\begin{bmatrix} \nu & n-\nu \\ n-\nu & \nu\end{bmatrix}\right)\right|$.
So
\[
|\mathfrak{S}_{2n}(\nu)|\, =\, \frac{(n!)^4}{(\nu!)^2((n-\nu)!)^2}\, .
\]
Let us use Stirling's formula: if we let $x=\nu/n$, then
\[
\begin{aligned}
&\frac{2^{-2n}(1+\delta)^{2\nu}(1-\delta)^{2n-2\nu}
|\mathfrak S_{2n}(\nu)|}{(n!)^2}\\
&\quad\sim \frac{e^{-2n\ln2}e^{2nx\ln(1+\delta)}}{2\pi nx(1-x)}\\
&\qquad\quad{}\times e^{2n(1-x)\ln(1-\delta)}
 e^{-2nx\ln x-2n(1-x)\ln(1-x)}.
\end{aligned}
\]
as $n \to \infty$ for all those $x \in (\epsilon,1-\epsilon)$ for some $\epsilon>0$.
(For other $x$, {\em a priori} bounds show that their contribution is negligible.)

Then we may define
\[
f(x)\, =\, -\ln(2)+x\ln(1+\delta)+(1-x)\ln(1-\delta)-x\ln(x)-(1-x)\ln(1-x)\, ,
\]
so that
\[
\frac{
2^{-2n}  (1+\delta)^{2\nu}(1-\delta)^{2n-2\nu} |\mathfrak{S}_{2n}(\nu)|}{(n!)^2}\,
\sim\, \frac{e^{2nf(x)}}{2\pi nx(1-x)}\, ,
\]
as $n \to \infty$.
Then setting $x_*=(1+\delta)/2$, we may see that
\[
f(x_*)\, =\, 0\, ,\qquad
f'(x_*)\, =\, 0\ \text{ and }\ 
f''(x_*)\, =\, -\frac{4}{1-\delta^2}\, .
\]
So, taking
$x=x_*+z/\sqrt{n}$, we have, by Taylor expansion,
\[
\frac{
2^{-2n}  (1+\delta)^{2\nu}(1-\delta)^{2n-2\nu} |\mathfrak{S}_{2n}(\nu)|}{(n!)^2}\,
\sim\, \frac{e^{-4z^2/(1-\delta^2)}}{\pi n (1-\delta^2)/2}\, .
\]

Note that ${T}^{(2)}=B$, which is symmetric, and 
\[
B^2\, =\, \frac{1}{2}\begin{bmatrix}1+\delta^2 & 1-\delta^2\\
1-\delta^2&1+\delta^2\end{bmatrix}
\]
So
\[
I^{(2)}+J^{(2)}-\left({T}^{(2)}\right)^2\, =\, \begin{bmatrix}1-(\delta^2/2) & (\delta^2/2)\\
(\delta^2/2) & 1-(\delta^2/2)\end{bmatrix}
\]
So $\operatorname{det}\left(I^{(2)}+J^{(2)}-\left({T}^{(2)}\right)^2\right)$
equals $1-\delta^2$.
If we let $z_{\nu} = (\nu-n(1+\delta)/2)/\sqrt{n}$,
then $\Delta z = 1/\sqrt{n}$.
Finally, note that we were supposed to divide by $(2n)!$ instead of $(n!)^2$.
So we obtain
\[
\sum_{\nu=0}^{n} \frac{
2^{-2n}  (1+\delta)^{2\nu}(1-\delta)^{2n-2\nu} |\mathfrak{S}_{2n}(\nu)|}{(2n)!}\,
\sim\, \frac{1}{\binom{2n}{n}} \sum_{\nu=0}^{n}\frac{e^{-4z_{\nu}^2/(1-\delta^2)}}{\pi n (1-\delta^2)/2}\, .
\]
Since $\binom{2n}{n} \sim 2^{2n}/\sqrt{\pi n}$, this gives
\[
\begin{aligned}
&\binom{2n}{n}^{-1}\sum_{\nu=0}^n
 \frac{e^{-4z_\nu^2/(1-\delta^2)}}{\pi n(1-\delta^2)/2}\\
&\quad\sim 2^{-2n}\sum_{\nu=0}^n
 \frac{e^{-4z_\nu^2/(1-\delta^2)}}{\sqrt\pi(1-\delta^2)/2}\,\Delta z\\
&\quad\sim 2^{-2n}\int_{-\infty}^\infty
 \frac{e^{-4z^2/(1-\delta^2)}}{\sqrt\pi(1-\delta^2)/2}\,dz
 =\frac{2^{-2n}}{\sqrt{1-\delta^2}},
\end{aligned}
\]
as $n \to \infty$.
This verifies the theorem, again.
Note that this shows the denominator $\sqrt{\mathcal{D}[\mathcal{C}]}$ is due to Gaussian fluctuations, as usual.
This is also standard in the theoretical physics community, where it goes by the name of ``1 loop correction''
to the tree (or $0$-loop approximation). See for example Section 3.4 of Itzykson and Drouffe, where it is applied
to the specific example of the Ising and $O(n)$ models using the $1/d$ expansion \cite{ItzyksonDrouffe}.

Before moving on from Example 2, let us consider the limiting case, $\delta=1$.
This means 
$B = \begin{bmatrix} 1 & 0 \\ 0 & 1 \end{bmatrix}$, instead of assuming all matrix entries are strictly positive.
Then we must have $\nu=n$.
So
\[
\frac{1}{(2n)!}\, \operatorname{perm}\left(A^{(2,n)}\right)\, =\, 1/\binom{2n}{n}\, ,
\]
so that
\[ 
\frac{1}{(2n)!}\, \operatorname{perm}\left(A^{(2,n)}\right)\,  \sim\, 2^{-2n}\, \sqrt{n \pi}\, .
\]
What goes wrong is that since some matrix entries of $B$ are $0$, we can have a collapse of the spectral gap.
Since $B={T}^{(2)}=I^{(2)}$, we have $\left({T}^{(2)}\right)^2=I^{(2)}$.
Then $I^{(2)}+J^{(2)}-\left({T}^{(2)}\right)^2=J^{(2)}$ which is a 
rank-1 projection. One eigenvalue is 1, and the other is $0$. So the determinant is $0$.

\begin{remark}
Our hypothesis for our main result is that all matrix entries of $B$ are strictly positive. In this case, by the Perron-Frobenius theorem,
there is a positive spectral gap below the eigenvalue $1$ of
${T}^{(m)}({T}^{(m)})^T$.
(See for example, \cite{Simon}, Theorem II.5.1.)
\end{remark}
This observation may be applied to the paper of Pal and that of Tripathi as well, because they assume an exponential form
for the kernel which is strictly positive.
Then the Krein-Rutman theorem applies in place of the finite dimensional Perron-Frobenius
theorem.

\section{Proof of the main result}

For $m=1$ the asserted formula is exact. We henceforth assume $m\geq2$.

\subsection{Asymptotics of Pinsky's formula}

We recall that Stirling's formula may be stated as
\[
e^{n\ln(n) - n}\, \sqrt{2\pi n}\, e^{1/(12n+1)}\, <\, n!\, <\, e^{n\ln(n) - n}\, \sqrt{2\pi n}\, e^{1/(12n)}\, .
\]
Using this, and equation (\ref{eq:Pinsky}) we may see that 
\[
\frac{1}{(m\cdot n)!}\, |\mathfrak{S}_{mn}(Q)|\, 
\sim\, \frac{\exp(n \mathcal{L}[Q/n])}{(2\pi n)^{(m-1)^2/2}}\, 
\mathcal{K}[Q/n]\, ,\ \text{ as $n \to \infty$,}
\]
under appropriate assumptions on $Q$, 
where for an $m\times m$ matrix $X = \big(x_{r,s}\, :\, r,s \in [m]\big)$, we have
\[
\mathcal{L}[X]\, =\, -m\ln(m) - \sum_{(r,s) \in [m]\times [m]} x_{r,s} \ln\left(x_{r,s}\right)\, ,
\]
and where
\[
\mathcal{K}[X]\, =\, m^{-1/2}  \left(\prod_{(r,s) \in [m] \times [m]} x_{r,s}\right)^{-1/2}\, .
\]
Now consider the formula
\[
\frac{1}{(m\cdot n)!}\, \operatorname{perm}\left(A^{(m,n)}\right)\,
=\,  \sum_{Q \in \mathcal{M}(m,n)}  \frac{|\mathfrak{S}_{mn}(Q)|}{(m\cdot n)!}\,
\exp\left(n\mathcal{P}_B[Q/n]\right)\, ,
\]
where $\mathcal{M}(m,n)$ is the set of matrices described before, and
\[
\mathcal{P}_B[X]\, =\, 
\sum_{r=1}^{m} \sum_{s=1}^{m} x_{r,s} \ln (b_{r,s})\, .
\]
The conditions on $Q$ are that $q_{r,s}\gg 1$ for all $(r,s)\in[m]^2$.
We may take a sequence $Q^{(n)} \in \mathcal{M}(m,n)$ (one for each $n$) such that
$Q^{(n)}/n \to {T}^{(m)}$.
This is possible because ${T}^{(m)}$ is a doubly stochastic matrix.
Note that all components of ${T}^{(m)}$ are strictly positive so that $Q^{(n)} \approx n {T}^{(m)}$ does
satisfy the conditions to apply Stirling's formula.
Hence,
\[
\begin{aligned}
&\mathcal L[Q^{(n)}/n]+\mathcal P_B[Q^{(n)}/n]\\
&\quad\longrightarrow -m\ln m-\sum_{r,s=1}^m t^{(m)}_{r,s}\ln(v_rw_s)\\
&\quad=-m\ln m-\sum_{r=1}^m\ln v_r-\sum_{s=1}^m\ln w_s,
\end{aligned}
\]
where the last equality follows from the fact that ${T}^{(m)}$
is doubly stochastic.

This gives the ``leading order'' behaviour. Now suppose that for some finite $n$ we have
\[
q_{r,s}\, =\, n t^{(m)}_{r,s} + n^{1/2} z_{r,s}\, ,
\]
for some matrix $Z = \big(z_{r,s}\, :\, r,s \in [m]\big)$ with row sums and column sums both $0$.
Then, Taylor expanding $\mathcal{L}$ to second order, we have
\[
n \mathcal{L}[Q/n] + n \mathcal{P}_B[Q/n]\, 
=\, n  \mathcal{L}[{T}^{(m)} + n^{-1/2} Z] + n \mathcal{P}_B[{T}^{(m)}+n^{-1/2} Z]\, ,
\]
which has first variation 0, and whose second variation gives, for $Z$ in a fixed bounded set, 
\[
\begin{aligned}
&n\mathcal L[T^{(m)}+n^{-1/2}Z]
 +n\mathcal P_B[T^{(m)}+n^{-1/2}Z]\\
&\quad=n\left(-m\ln m-\sum_{r=1}^m\ln v_r-\sum_{s=1}^m\ln w_s\right)\\
&\qquad{}-\sum_{r,s=1}^m\frac{z_{r,s}^2}{2t^{(m)}_{r,s}}+O(n^{-1/2}).
\end{aligned}
\]
Now we want to account for the terms in the last summation: matrices $Q = n{T}^{(m)} + n^{1/2} Z$ where
$Z$ has row-sums and column-sums all equal to $0$.

We will define a $(m-1)\times (m-1)$ matrix $\Xi = \big(\xi_{\rho,\sigma}\, :\, \rho,\sigma \in [m-1]\big)$
with arbitrary entries. We can let $z_{\rho,\sigma} = \xi_{\rho,\sigma}$ for $(\rho,\sigma) \in [m-1]^2$.
But then we adjust the values $z_{\rho,m}$, $z_{m,\sigma}$ and $z_{m,m}$
so that $Z$ has all row-sums and column-sums equal to $0$.
For each $(\rho,\sigma) \in [m-1]^2$, define the matrix $F^{(\rho,\sigma)}$ where
\[
F^{(\rho,\sigma)}_{r,s}\, =\, \begin{cases} 1 & \text{ if $(r,s) = (\rho,\sigma)$,}\\
-1 & \text{ if $(r,s) =(\rho,m)$ or $(r,s) = (m,\sigma)$,}\\
1 & \text{ if $(r,s) = (m,m)$,}\\
0 & \text{ otherwise,}
\end{cases}
\]
so that $F^{(\rho,\sigma)}$ has all row-sums and all column-sums equal to $0$.
Then we consider 
\[
Z\, =\, \sum_{(\rho,\sigma) \in [m-1]^2} \xi_{\rho,\sigma} F^{(\rho,\sigma)}\, .
\]
Let $\mathcal{G} : \R^{m\times m} \to \R$ be the positive-definite quadratic form
\[
\mathcal{G}(Z)\, =\,  \sum_{(r,s) \in [m]^2} \frac{z_{r,s}^2}{t^{(m)}_{r,s}}\, ,
\]
and let $\mathfrak{G} = \big(\mathfrak{g}_{(\rho,\sigma),(\rho',\sigma')}\, :\, (\rho,\sigma),(\rho',\sigma') \in [m-1]^2\big)$
be the positive-definite matrix such that 
\[
\mathcal{G}\left(\sum_{(\rho,\sigma) \in [m-1]^2} \xi_{\rho,\sigma} F^{(\rho,\sigma)}\right)\,
=\, \sum_{(\rho,\sigma) \in [m-1]^2} \sum_{(\rho',\sigma') \in [m-1]^2} 
\mathfrak{g}_{(\rho,\sigma),(\rho',\sigma')} \xi_{\rho,\sigma} \xi_{\rho',\sigma'}\, .
\]
Let $\widetilde{\mathcal{G}} : \R^{(m-1)\times (m-1)} \to \R$ be the positive-definite quadratic form
\[
\widetilde{\mathcal{G}}(\Xi)\, =\, \sum_{(\rho,\sigma) \in [m-1]^2} \sum_{(\rho',\sigma') \in [m-1]^2} 
\mathfrak{g}_{(\rho,\sigma),(\rho',\sigma')} \xi_{\rho,\sigma} \xi_{\rho',\sigma'}\, .
\]
Then, using the rigorous Stirling bounds 
\[
\begin{aligned}
&\frac{\operatorname{perm}(A^{(m,n)})}{(mn)!}\\
&\quad=(1+O(n^{-1/2}))\,
 \frac{m^{-mn}(\prod_{r=1}^m v_r)^{-n}(\prod_{s=1}^m w_s)^{-n}}
 {m^{1/2}(\prod_{r,s=1}^m t^{(m)}_{r,s})^{1/2}}\\
&\qquad{}\times\sum_{\Xi\in\widetilde M(m,n)}
 \frac{\exp(-\widetilde{\mathcal G}(\Xi)/2)}{(2\pi n)^{(m-1)^2/2}},
\end{aligned}
\]
where $\widetilde{M}(m,n)$ is the set of matrices $\Xi$ such that $Q^{(n)} = n {T}^{(m)} + n^{1/2} Z$
satisfies $Q^{(n)} \in \mathcal{M}(m,n)$ for $Z$ being the matrix 
equal to $\sum_{(\rho,\sigma) \in [m-1]^2} \xi_{\rho,\sigma} F^{(\rho,\sigma)}$.
Note that the conditions $q^{(n)}_{r,s} \in (0,n)$ are automatically satisfied for $\|Z\|_\infty=o(n^{1/2})$
because ${T}^{(m)}$ is assumed to have no zero entries (hence also no $1$ entries by the double-stochastic condition).
The other condition is $q^{(n)}_{r,s} \in \Z$, which forces 
\begin{equation}
\label{eq:preconditions}
z_{r,s} \in \left\{-n^{1/2} t^{(m)}_{r,s} + n^{-1/2} q_{r,s}\, :\, q_{r,s} \in \Z\right\}\, ,
\end{equation}
for each $(r,s) \in [m]^2$.
This forces 
\begin{equation}
\label{eq:conditions}
\xi_{\rho,\sigma} \in \left\{-n^{1/2} t^{(m)}_{\rho,\sigma} + n^{-1/2} q_{\rho,\sigma}\, :\, q_{\rho,\sigma} \in \Z\right\}\, ,
\end{equation}
for each $(\rho,\sigma) \in [m-1]^2$. But since each $F^{(\rho,\sigma)}$ has integer components, if the conditions of (\ref{eq:conditions}) are satisfied
for each $(\rho,\sigma) \in [m-1]^2$, then that does imply conditions (\ref{eq:preconditions}) for each $(r,s) \in [m]^2$.

In turn, this all means that $\Delta \xi_{\rho,\sigma} = n^{-1/2}$ for each $(\rho,\sigma) \in [m-1]^2$.
Therefore,
\begin{equation}
\begin{aligned}
&\lim_{n\to\infty}
 \frac{m^{mn}(\prod_{r=1}^m v_r)^n(\prod_{s=1}^m w_s)^n}{(mn)!}
 \operatorname{perm}(A^{(m,n)})\\
&\quad=\frac{1}{m^{1/2}(\prod_{r,s=1}^m t^{(m)}_{r,s})^{1/2}}\\
&\qquad{}\times\lim_{n\to\infty}\sum_{\Xi\in\widetilde M(m,n)}
 \frac{\exp(-\widetilde{\mathcal G}(\Xi)/2)}{(2\pi)^{(m-1)^2/2}}
 \prod_{\rho,\sigma=1}^{m-1}\Delta\xi_{\rho,\sigma}\\
&\quad=\frac{1}{m^{1/2}(\prod_{r,s=1}^m t^{(m)}_{r,s})^{1/2}}\\
&\qquad{}\times\int_{\R^{(m-1)\times(m-1)}}
 \frac{\exp(-\widetilde{\mathcal G}(\Xi)/2)}{(2\pi)^{(m-1)^2/2}}
 \prod_{\rho,\sigma=1}^{m-1}d\xi_{\rho,\sigma}.
\end{aligned}
\end{equation}
In other words, performing the Gaussian integral,
\begin{equation}
\label{eq:LDP}
\begin{aligned}
&\lim_{n\to\infty}
 \frac{m^{mn}(\prod_{r=1}^m v_r)^n(\prod_{s=1}^m w_s)^n}{(mn)!}
 \operatorname{perm}(A^{(m,n)})\\
&\quad=\frac{1}{m^{1/2}(\prod_{r,s=1}^m t^{(m)}_{r,s})^{1/2}}
 \cdot\frac{1}{\sqrt{\det\mathfrak G}}.
\end{aligned}
\end{equation}

\subsection{Combinatorial simplification of the determinant}

The main contribution here is the following exact result.
\begin{lemma}
\label{lem:mat}
With the set-up as above,
\[
\begin{gathered}
\frac{1}{m^{1/2}(\prod_{r,s=1}^m t^{(m)}_{r,s})^{1/2}}
 \cdot\frac{1}{\sqrt{\det\mathfrak G}}\\
=\frac{1}{\sqrt{\det(I^{(m)}+J^{(m)}-T^{(m)}(T^{(m)})^T)}}.
\end{gathered}
\]
\end{lemma}
When combined with equation (\ref{eq:LDP}), this proves Theorem \ref{thm:main}.
We will prove this lemma in the rest of this section.
Firstly, note that
\[
\int_{\R^{m\times m}} \frac{\exp\left(-\mathcal{G}(Z)/2\right)}{(2\pi)^{m^2/2}}\,
\prod_{r=1}^{m} \prod_{s=1}^{m} dz_{r,s}\, =\, 
\left(\prod_{r=1}^{m} \prod_{s=1}^{m} t^{(m)}_{r,s}\right)^{1/2}\, ,
\]
because 
\[
\exp\left(-\mathcal{G}(Z)/2\right)\, =\, \prod_{r=1}^{m} \prod_{s=1}^{m}
\exp\left(-\frac{z_{r,s}^2}{2t^{(m)}_{r,s}}\right)\, .
\]
So this is the standard Gaussian integral calculation: it just happens that $\mathcal{G}$ is already diagonal in 
the standard basis.

But now we are more interested in the Gaussian integral
\[
\int_{\R^{(m-1)\times (m-1)}}
\frac{\exp\left(-\widetilde{\mathcal{G}}(\Xi)/2\right)}
{(2\pi)^{(m-1)^2/2}}\, 
\prod_{(\rho,\sigma) \in [m-1]^2} d\xi_{\rho,\sigma}\,
=\,  \sqrt{\det(\mathfrak{G}^{-1})}\, .
\]
Compared to the diagonal quadratic form $\mathcal{G}$, we have
\[
\widetilde{\mathcal{G}}(\Xi)\, =\, \mathcal{G}\left(\sum_{(\rho,\sigma) \in [m-1]^2} \xi_{\rho,\sigma} F^{(\rho,\sigma)}\right)\, ,
\]
so that we are conjugating by a linear transformation.
But we may think of it in another way.

Let $\mathsf{Z}$ be the random matrix in $\R^{m\times m}$ with independent
matrix entries, such that $(\mathsf{Z})_{r,s}$ is a centred normal
$\mathcal{N}(0,t^{(m)}_{r,s})$ 
random variable for each $(r,s) \in [m]^2$, meaning the variance
is $t^{(m)}_{r,s}$.
So we have
\[
\mathbf{P}\left(\mathsf{Z}\in dZ\right)\, =\, 
\frac{\exp\left(-\mathcal{G}(Z)/2\right)}
{(2\pi)^{m^2/2}\, \left(\prod_{r=1}^{m} \prod_{s=1}^{m} t^{(m)}_{r,s}\right)^{1/2}}\, 
\prod_{r=1}^m\prod_{s=1}^m dz_{r,s}\, .
\]

For each $r \in [m]$, let $\mathsf{R}_{r}$ be the random variable for the row sum
\[
\mathsf{R}_{r}\, =\, \sum_{s=1}^{m} (\mathsf{Z})_{r,s}\, ,
\]
and for each $\sigma \in [m-1]$, let $\mathsf{C}_{\sigma}$ be the random variable for the column sum
\[
\mathsf{C}_{\sigma}\, =\, \sum_{r=1}^{m} (\mathsf{Z})_{r,\sigma}\, .
\]
Note that we do not need to define $\mathsf{C}_m$ since $\mathsf{C}_m=\mathsf{R}_1+\dots+\mathsf{R}_m-(\mathsf{C}_1+\dots+\mathsf{C}_{m-1})$.
Only $2m-1$ of the row-sums and column-sums are linearly independent.

Then
\[
\begin{aligned}
&\mathbf P\Bigl(((\mathsf Z)_{\rho,\sigma}:
 (\rho,\sigma)\in[m-1]^2)\in d\Xi\,\Big|\,\mathcal E_0\Bigr)\\
&\quad=\frac{\exp(-\widetilde{\mathcal G}(\Xi)/2)}
 {(2\pi)^{(m-1)^2/2}\sqrt{\det(\mathfrak G^{-1})}}
 \prod_{\rho,\sigma=1}^{m-1}d\xi_{\rho,\sigma},
\end{aligned}
\]
where $\mathcal{E}_0$ is the event
\[
\mathcal{E}_0\, =\, \bigcap_{r\in [m]} \{\mathsf{R}_{r}=0\} \cap \bigcap_{\sigma \in [m-1]} \{\mathsf{C}_{\sigma}=0\}\, .
\]
Indeed, this is by construction, since we started with a discrete approximation of the jointly Gaussian random variables
and imposed the condition by hand.
But now we use the well-known result that another way to obtain the conditional distribution
is to use the Schur complement formula.

Using the conditional distribution formula and the Schur complement formula for the determinant of a block matrix,
we have derived the following:
\[
\det\left(\mathfrak{G}^{-1}\right)\, =\, \frac{\prod_{r=1}^{m} \prod_{s=1}^{m} t_{r,s}^{(m)}}{\det(\mathfrak{H})}\, ,
\]
where $\mathfrak H^{-1}$ is the precision matrix (inverse of the covariance) for
\[
(\mathsf R_1,\ldots,\mathsf R_m,\mathsf C_1,\ldots,\mathsf C_{m-1}).
\]
See for example \cite{Muirhead} Theorem 1.2.11 for a review 
of the conditional distribution.

Because of this, Lemma \ref{lem:mat} will be proved if we establish
\[
\det(\mathfrak{H})\, =\, \frac{1}{m} \det\left(I^{(m)}+J^{(m)} 
-  {T}^{(m)} \left({T}^{(m)}\right)^T\right)\, .
\]
Instead of calculating the determinant of $\mathfrak{H}$ directly, 
let us make a slightly different matrix $\mathfrak{K}$ which is a $(2m)\times (2m)$ matrix.
Let us define random variables $\mathsf{X}_1,\dots,\mathsf{X}_{2m}$ such that
\[
\forall  r \in [m]\, ,\ \text{ we have }\
\mathsf{X}_r\ =\, \mathsf{R}_r\ \text{ and }\
\mathsf{X}_{m+r}\, =\, -\mathsf{C}_r\, .
\]
Then we let $\mathfrak{K}$ be the covariance matrix for this vector. By inspection
\[
\mathfrak{K}\, =\, \begin{bmatrix} I^{(m)} & -{T}^{(m)} \\ -\left({T}^{(m)}\right)^T & I^{(m)} \end{bmatrix}\, .
\]
Note that the row-sums and column-sums of $\mathfrak{K}$ are all $0$. Therefore, by a well-known property of the adjugate
matrix, we have
\[
\operatorname{adj}(\mathfrak{K})\, =\, (2m) c J^{(2m)}\, ,
\]
for some constant $c$.
This fact for adjugate matrices is usually discussed in the context of Kirchhoff's matrix-tree
theorem because the graph Laplacian does have all row-sums and column-sums equal to $0$.

In other words, for each $i,j\in[2m]$ the $(i,j)$ cofactor
\[
C_{i,j}=(-1)^{i+j}\det\left(\bigl((\mathfrak K)_{a,b}\bigr)_{
\substack{a\in[2m]\setminus\{i\}\\b\in[2m]\setminus\{j\}}}\right)
\]
is equal to $c$.
But it is easy to see that the $(2m,2m)$ cofactor $C_{2m,2m}$ is equal to $\det(\mathfrak{H})$ because, defining
${U}$ to be the $m$-by-$(m-1)$ rectangular matrix such that
\[
({U})_{r,\sigma} = ({T}^{(m)})_{r,\sigma}\, ,
\]
(so that it is just the truncation), we have by inspection
\[
\mathfrak{H}\, =\, \begin{bmatrix} I^{(m)} & {U} \\ {U}^T & I^{(m-1)} \end{bmatrix}\, ,
\]
and so
\[
\mathfrak{H}\, =\, \begin{bmatrix} I^{(m)} & 0 \\ 0 &-I^{(m-1)} \end{bmatrix} 
\Big((\mathfrak{K})_{i,j} : (i,j) \in [2m-1]^2\Big) \begin{bmatrix} I^{(m)} & 0 \\ 0 &-I^{(m-1)} \end{bmatrix}\, .
\]
The conjugation by these diagonal matrices does not change the determinant because they are self-inverse
(involutory) because $(-1)^{m-1} (-1)^{m-1}=1$.
So
\[
\det(\mathfrak{H})\, =\, c\, .
\]

It remains to calculate the constant $c$. 
We can do this by fully describing the spectrum of $\mathfrak{K}$.
Note that if $\boldsymbol{\phi}$ is an eigenvector of ${T}^{(m)} \left({T}^{(m)}\right)^T$ with eigenvalue $\lambda>0$,
then, setting $\boldsymbol{\psi} = \lambda^{-1/2} \left({T}^{(m)}\right)^T \boldsymbol{\phi}$, we have
\[
\mathfrak{K} \begin{bmatrix} \boldsymbol{\phi} \\ -\boldsymbol{\psi} \end{bmatrix}\,
=\, (1+\lambda^{1/2}) \begin{bmatrix} \boldsymbol{\phi} \\ -\boldsymbol{\psi} \end{bmatrix}\ \text{ and }
\mathfrak{K} \begin{bmatrix} \boldsymbol{\phi} \\ \boldsymbol{\psi} \end{bmatrix}\,
=\, (1-\lambda^{1/2}) \begin{bmatrix} \boldsymbol{\phi} \\ \boldsymbol{\psi} \end{bmatrix}\, .
\]
(To account for $\lambda=0$ we can pair up null vectors of 
$\left({T}^{(m)}\right)^T$, $\boldsymbol{\phi}$,
with null vectors of ${T}^{(m)}$, $\boldsymbol{\psi}$,
arbitrarily. They are equal in number by the rank-nullity theorem, as a matrix and its transpose share the same rank.)
Write $\boldsymbol{\chi}^{(k)}=k^{-1/2}(1,\ldots,1)^T$. Then $\boldsymbol{\chi}^{(m)}$ is an eigenvector of ${T}^{(m)}\left({T}^{(m)}\right)^T$ with eigenvalue $\lambda=1$.
So (by the Perron-Frobenius theorem for  ${T}^{(m)}\left({T}^{(m)}\right)^T $) the normalised positive null-vector of $\mathfrak{K}$
is 
\[
\boldsymbol{\chi}^{(2m)}\, .
\]
So if we take $\mathfrak{K}+J^{(2m)}$, then its determinant will be
\[
(0+1) \cdot 2 \cdot \left(\prod_{r=2}^{m} (1+\sqrt{\lambda_r})\right) \left(\prod_{r=2}^{m} (1-\sqrt{\lambda_r})\right)\, ,
\]
where $1=\lambda_1>\lambda_2\geq \dots \geq \lambda_m\geq 0$ are the eigenvalues of ${T}^{(m)}\left({T}^{(m)}\right)^T$.
But the determinant of $\mathfrak{K}$ itself is $0$.
So by the ``matrix determinant lemma,''
\[
2 \prod_{r=2}^{m} (1-\lambda_r)\, =\, \left(\boldsymbol{\chi}^{(2m)}\right)^T \operatorname{adj}(\mathfrak{K}) \boldsymbol{\chi}^{(2m)}\, .
\]
So we have
\[
2 \prod_{r=2}^{m} (1-\lambda_r)\, =\, (2m)c\left(\boldsymbol{\chi}^{(2m)}\right)^T J^{(2m)} \boldsymbol{\chi}^{(2m)}\, .
\]
But $J^{(2m)} \boldsymbol{\chi}^{(2m)}=\boldsymbol{\chi}^{(2m)}$, and $\boldsymbol{\chi}^{(2m)}$ is normalised.
So
\[
c\, =\, \frac{1}{m}\, \prod_{r=2}^{m} (1-\lambda_r)\, ,
\]
as desired.

\subsection{Probabilistic interpretation as a local central limit theorem}

The asymptotic evaluation of the permanent in our proof has a natural probabilistic interpretation in terms of a local central limit theorem (LCLT) for the Gibbs measure on the symmetric group already mentioned in the introduction.

More precisely, given a block-uniform matrix $A^{(m,n)}$, we can define a probability measure on $S_{mn}$ where the probability of a permutation $\pi \in S_{mn}$ is proportional to its weight in the permanent:
\[
\mathbb{P}(\pi)\, =\, \frac{1}{\operatorname{perm}(A^{(m,n)})} \prod_{i=1}^{mn} a_{i,\pi_i}^{(m,n)}\, .
\]
Because the matrix $A^{(m,n)}$ is block-uniform, the weight of $\pi$ depends only on its macroscopic block-to-block transition counts. Let $Q(\pi) \in \mathcal{M}(m,n)$ be the random block-count matrix, where $q_{r,s}(\pi)$ is the number of elements in block $r$ mapped to block $s$ by $\pi$. The permanent acts exactly as the partition function for this measure. 

Under this measure, the sequence of random matrices $Q(\pi)/n$ satisfies a law of large numbers, converging to the optimal doubly stochastic matrix $T^{(m)}$ given by the Schr\"odinger bridge equations. Furthermore, the saddle-point expansion we performed in Section 3.1 demonstrates that the scaled fluctuations
\[
Z^{(n)}\, =\, \frac{Q(\pi) - n T^{(m)}}{\sqrt{n}}
\]
satisfy a local central limit theorem. Specifically, $Z^{(n)}$ takes values on a translated lattice in the $(m-1)^2$-dimensional linear subspace of matrices with zero row and column sums. Put $d=(m-1)^2$ and let $\Xi(q)$ denote the $(m-1)\times(m-1)$ upper-left block of $(q-nT^{(m)})/\sqrt n$. The lattice-normalised statement is
\[
n^{d/2}\mathbb P(Q=q)
=\frac{\sqrt{\det\mathfrak G}}{(2\pi)^{d/2}}
  \exp\left(-\frac12\widetilde{\mathcal G}(\Xi(q))\right)+o(1),
\]
uniformly over feasible $q$ for which $\Xi(q)$ stays in a fixed compact set. 

In this probabilistic framework, the leading exponential term $n \Lambda[\mathcal{C}]$ in Pal's conjecture relates to the law of large numbers, while the algebraic prefactor $1/\sqrt{\mathcal{D}[\mathcal{C}]}$ combines the Gaussian integral with the Stirling prefactor. Our derivation of the Fredholm determinant is therefore equivalent to evaluating the determinant of the covariance matrix of these Gaussian fluctuations around the optimal transport plan.

\appendix

\section{Complex block-uniform matrices and the saddle point method}

Our result is just for the special case of block-uniform functions.
In the case of twice differentiable functions, the proof of Pal's conjecture has been addressed already by Raghavendra Tripathi
using an investigation by Peter McCullagh.
We feel that our derivation for the special case has some pedagogic value.
In addition, these special matrices, the block-uniform matrices, may be useful
for studying other phenomena.
Here we consider the analogous question for block-uniform matrices with complex entries.

As before, consider a fixed, finite $m \in \N$ and let $B$ be an $m\times m$ matrix $B = \big(b_{r,s}\, :\, (r,s) \in [m]^2\big)$ such that
\[
\forall (r,s) \in [m]^2\, ,\ \text{ we have }\ b_{r,s} \in \C\, .
\]
Previously we had assumed all entries were positive. 
Now we allow the entries to be complex.
For each $n \in \N$, let us define the matrix $A^{(m,n)}$ such that $A^{(m,n)} = \big(a^{(m,n)}_{i,j}\, :\, (i,j) \in [mn]^2\big)$
given by the formula
\[
a_{i,j}^{(m,n)}\, =\, b_{\lceil i/n\rceil, \lceil j/n \rceil}\, .
\]

An important problem in mathematics is understanding the complexity of calculating the permanent for matrices with complex entries.
For example, for complex, Ginibre ensemble matrices, the conjecture of Aaronson and Arkhipov is that the problem is \#P-hard.
A more recent article that reiterates that basic conjecture, while improving the setup for the quantum Boson sampling algorithm
is \cite{Huh}, which also uses Ryser's formula
\[
\operatorname{perm}(A)\, =\, \sum_{\boldsymbol{y} \in \{0,1\}^n} (-1)^{n+\|\boldsymbol{y}\|^2} \prod_{j=1}^{n} \left(\boldsymbol{a}^{(j)} \cdot \boldsymbol{y}\right)\, ,
\]
where the $n\times n$ matrix $A = (a_{j,k}\, :\, (j,k) \in [n]^2)$ can be written as a row of column vectors $(\boldsymbol{a}^{(1)},\dots,\boldsymbol{a}^{(n)})$.
In fact, Huh invents his own formula for the permanent based on this, which is important for his improvements to the set-up of Aaronson
and Arkhipov.
One might note the similarity between a formula such as this one and the computational complexity
problems that are amenable to study using spin-glass techniques.

\subsection{The saddle point method conjecture}

We can write
\[
\frac{1}{(mn)!}\, \operatorname{perm}(A^{(m,n)})\,
=\, \binom{mn}{n,\dots,n}^{-1} (n!)^m \sum_{Q \in \mathcal{M}(m,n)} 
\left( \prod_{(r,s) \in [m]^2} \left(\frac{b_{r,s}^{q_{r,s}}}{q_{r,s}!}\right)\right)\, ,
\]
which could be written as 
\[
\begin{aligned}
&\frac{\operatorname{perm}(A^{(m,n)})}{(mn)!}\\
&\quad=\binom{mn}{n,\ldots,n}^{-1}(n!)^m
 [(x_1\cdots x_m y_1\cdots y_m)^n]\\
&\qquad\left\{\exp\left(\sum_{r,s=1}^m b_{r,s}x_ry_s\right)\right\}.
\end{aligned}
\]
Therefore, using Cauchy's integral formula, one can write
\begin{equation}
\begin{aligned}
&\frac{\operatorname{perm}(A^{(m,n)})}{(mn)!}\\
&\quad=\binom{mn}{n,\ldots,n}^{-1}(n!)^m
 \int_{[-\pi,\pi]^{2m}}\\
&\qquad F_{n,B}\bigl(x_1e^{i\theta_1},\ldots,x_me^{i\theta_m},
 y_1e^{i\phi_1},\ldots,y_me^{i\phi_m}\bigr)\\
&\qquad{}\times\left(\prod_{r=1}^m\frac{d\theta_r}{2\pi}\right)
 \left(\prod_{s=1}^m\frac{d\phi_s}{2\pi}\right),
\end{aligned}
\end{equation}
where
\[
F_{n,B}(x_1,\dots,x_m,y_1,\dots,y_m)\,
=\, 
\frac{ \exp\left(\sum_{(r,s) \in [m]^2} b_{r,s} x_r y_s\right)}
{(x_1\cdots x_m y_1\cdots y_m)^n}\, .
\]
That suggests seeking the critical points of $F_{n,B}$: these give the following generalisation of the bridge equations
\[
\begin{aligned}
\sum_{s=1}^m b_{r,s}x_ry_s&=n &&\text{for every }r\in[m],\\
\sum_{r=1}^m b_{r,s}x_ry_s&=n &&\text{for every }s\in[m],
\end{aligned}
\]
where $\boldsymbol{x}$ and $\boldsymbol{y}$ are complex vectors in $\C^m$.
Letting $v_r = x_r/\sqrt{n}$ and $w_s = y_s/\sqrt{n}$, then we
recover
(\ref{eq:BridgeV}) and (\ref{eq:BridgeW}).
But now we seek all complex solutions.

Now considering the integral again, we have
\[
\begin{aligned}
&\frac{\operatorname{perm}(A^{(m,n)})}{(mn)!}\\
&\quad=\binom{mn}{n,\ldots,n}^{-1}(n!)^m
 F_{n,B}(\boldsymbol x,\boldsymbol y)\\
&\qquad{}\times\int_{[-\pi,\pi]^{2m}}
 e^{nG(\boldsymbol\theta,\boldsymbol\phi)}
 \left(\prod_{r=1}^m\frac{d\theta_r}{2\pi}\right)
 \left(\prod_{s=1}^m\frac{d\phi_s}{2\pi}\right),
\end{aligned}
\]
where
\[
G(\boldsymbol{\theta},\boldsymbol{\phi})\, =\, \sum_{r=1}^{m} \sum_{s=1}^{m} b_{r,s} v_r w_s (e^{i(\theta_r+\phi_s)}-1)-i\sum_{r=1}^{m} \theta_r - i\sum_{s=1}^{m} \phi_s\, .
\]
Using the bridge equations and Taylor expanding to second order gives
\[
G(\boldsymbol{\theta},\boldsymbol{\phi})\, =\, -\frac{1}{2}\, g(\boldsymbol{\theta},\boldsymbol{\phi})+o(\|\boldsymbol{\theta}\|^2+\|\boldsymbol{\phi}\|^2)\, ,
\]
where
\[
g(\boldsymbol{\theta},\boldsymbol{\phi})\, =\, 
\sum_{r=1}^{m} \sum_{s=1}^{m} b_{r,s} v_r w_s (\theta_r+\phi_s)^2\, =\, \begin{bmatrix} \boldsymbol{\theta}^T\, ,\ \boldsymbol{\phi}^T \end{bmatrix}
\begin{bmatrix} I^{(m)} & T^{(m)} \\ 
\left(T^{(m)}\right)^T & I^{(m)} \end{bmatrix} \begin{bmatrix} \boldsymbol{\theta}\\ \boldsymbol{\phi}\end{bmatrix}\, .
\]
Of course, we continue to observe a zero-mode corresponding to the symmetry $\boldsymbol{v} \leftarrow \lambda \boldsymbol{v}$
and
 $\boldsymbol{w} \leftarrow \lambda^{-1} \boldsymbol{w}$ for $\lambda \in \C \setminus \{0\}$.
But removing this zero-mode by hand, we see that the condition to have a stable critical point is that
\begin{equation}
\label{eq:Matrix}
R^{(2m)}\, =\, 
\begin{bmatrix} I^{(m)}+\frac{1}{2}\, J^{(m)} & T^{(m)}-\frac{1}{2}\, J^{(m)} \\ 
\left(T^{(m)}\right)^T-\frac{1}{2}\, J^{(m)} & I^{(m)}+\frac{1}{2}\, J^{(m)}  \end{bmatrix} 
\end{equation}
is a strictly accretive matrix, meaning that its numerical range is in the right half-plane $\{x+iy\, :\, x \in (0,\infty)\, ,\ y \in \R\}$.
Then, if we were to apply the standard Hayman saddle-point method, as in \cite{FlajoletSedgewick} we would 
obtain an asymptotic expression similar to Theorem \ref{thm:main}.
But it is true that we would need to verify the other technicalities of the Hayman method, such as verifying that the 
minor arcs have integrals which are negligible in the asymptotic limit.

With this as background, let us now make the following definition.
\begin{def}
\label{def:1}
(a) Let $\mathcal{M}_1$ be the set of matrices $Z = \big(z_{r,s}\, :\, (r,s) \in [m]^2\big)$ of the form
\[
z_{r,s}\, =\, v_r w_s\, ,
\]
for $\boldsymbol{v},\boldsymbol{w} \in \C^m$ satisfying the following: the matrix $T^{(m)}(Z) = \big(t^{(m)}_{r,s}\, :\, (r,s) \in [m]^2\big)$,
given by
\[
t_{r,s}^{(m)}\, =\, b_{r,s} z_{r,s}\, ,
\]
has all row-sums and all column-sums equal to $1$. 

(b) For each $Z \in \mathcal{M}_1$, define $\mathcal{W}(Z) = 1/(z_{1,1}\cdots z_{m,m})$ to be the weight of $Z$.

(c) For $Z \in \mathcal{M}_1$, if $|\mathcal{W}(Z)| = \max_{Z' \in \mathcal{M}_1} |\mathcal{W}(Z')|$ then we say
$Z$ has a dominant weight.

(d) If the following condition is satisfied for each $Z \in \mathcal{M}_1$ which has a dominant weight, then we will say
$B$ is of ``simple type.'' The condition is that the Hessian replacement matrix (accounting for the zero-mode):
\[
R^{(2m)}(Z)\, \stackrel{\mathrm{def}}{:=}\,
\begin{bmatrix} I^{(m)}+\frac{1}{2}\, J^{(m)}  & T^{(m)}(Z) - \frac{1}{2}\, J^{(m)}\\
\left(T^{(m)}(Z)\right)^T - \frac{1}{2}\, J^{(m)} & I^{(m)} + \frac{1}{2}\, J^{(m)}\end{bmatrix}
\] 
is a strictly accretive matrix,
meaning the numerical range is a subset of the right half-plane $\{x+iy\, :\, x \in (0,\infty)\, ,\ y \in \R\}$.
\end{def}

\begin{conj}
\label{conj:1}
Suppose that $B$ is an $m \times m$ matrix of ``simple type'' according to the definition above.
Define $\mathcal{M}_2 \subset \mathcal{M}_1$ to be the subset of those matrices $Z \in \mathcal{M}_1$ which 
each have dominant weight.
Writing $W_* = \max_{Z\in\mathcal M_1}|\mathcal W(Z)|$, we conjecture that
\[
\begin{aligned}
\frac{\operatorname{perm}(A^{(m,n)})}{(mn)!}
&=\sum_{Z\in\mathcal M_2}
\frac{\sqrt2\,m^{-mn}\mathcal W(Z)^n}
{\prod_{k=1}^{2m}\sqrt{\lambda_k(Z)}}\\
&\quad+o\!\left(m^{-mn}W_*^n\right),\qquad n\to\infty,
\end{aligned}
\]
where $\lambda_1(Z),\dots,\lambda_{2m}(Z)$ are the eigenvalues of $R^{(2m)}(Z)$, all in the open right half-plane, and the square roots are the principal square roots. The additive error formulation allows cancellation between dominant contributions.
\end{conj}
The factor $\sqrt2$ accounts for the common constant mode: with the definition of $R^{(2m)}$ above,
\[
\det R^{(2m)}(Z)
=2\det\!\left(I^{(m)}+J^{(m)}-T^{(m)}(Z)T^{(m)}(Z)^T\right).
\]
In particular, for $m=1$ the eigenvalues are $1,2$, and the corrected prefactor is $1$.
In principle, the conjecture above is true as long as one can establish bounds for the integrand on the ``minor arcs.''
This is the standard Hayman method, as in 
Flajolet and Sedgewick \cite{FlajoletSedgewick}.

A careful analysis of the asymptotics of an integral as we have stated it should start from a reference
such as Pham \cite{Pham}.
Another more recent textbook is \cite{PemantleWilsonMelczer}.
It may be a topic for further study to perform that careful analysis.

\subsection{Main example}
Consider the one-parameter complex deformation of the $2\times 2$ block-uniform matrix, where the off-diagonal entries are replaced by a parameter $z \in \mathbb{C}$:
\[
B(z)\, =\, \begin{bmatrix} 1 & z \\ z & 1 \end{bmatrix}\, .
\]
In order to simplify the calculations let us make a small change.
Note that for any matrix $B$ if we define $\widetilde{B}$ by the formula
\[
(\widetilde{B})_{r,s}\, =\, (B)_{r,s} \mathfrak{X}_r \mathfrak{Y_s}\, ,
\]
then we have
\[
\operatorname{perm}\big(\widetilde{B} \otimes (n J^{(n)})\big)\, =\, (\mathfrak{X}_1\cdots \mathfrak{X}_m \cdot \mathfrak{Y}_1 \cdots \mathfrak{Y}_m)^n
\operatorname{perm}\big(B\otimes (nJ^{(n)})\big)\, .
\]
For $z\ne0$, if we use $(\mathfrak{X}_1,\mathfrak{X}_2) = (\mathfrak{Y}_1,\mathfrak{Y}_2) = (1,1/z)$ we obtain
\[
\widetilde{B}(z)\, =\, \begin{bmatrix} 1 & 1 \\ 1 & 1/z^2 \end{bmatrix}
\]
Now, let us also denote $\xi = 1/z$.
Then this
is an explicitly calculable model because for $A^{(2,n)} = \widetilde{B} \otimes (n J^{(n)})$, we have
\[
\frac{1}{(2n)!}\, \operatorname{perm}(A^{(2,n)})\, =\, \binom{2n}{n}^{-1} \sum_{k=0}^{n} \binom{n}{k}^2 \xi^{2k}\,
=\, \binom{2n}{n}^{-1} (1-\xi^2)^n P_n\left(\frac{1+\xi^2}{1-\xi^2}\right)\, ,
\]
where $P_n(z)$ is the Legendre polynomial
\[
P_n(z)\, =\, \frac{1}{2^n}\, \sum_{k=0}^{n} \binom{n}{k}^2(z+1)^k(z-1)^{n-k}\, .
\]
For $\xi\notin\{0,1,-1\}$, we have $\mathcal{M}_1=\{Z_+,Z_-\}$ where
\[
Z_{\pm} =\, \frac{1}{\xi^2\pm \xi} \begin{bmatrix} \xi^2 & \pm \xi \\ \pm \xi & 1 \end{bmatrix}\, ,\ \text{ so }\
T^{(2)}(Z_{\pm}) =\, \frac{1}{\xi\pm 1} \begin{bmatrix} \xi & \pm 1 \\ \pm 1 & \xi \end{bmatrix}\, .
\]
This gives $\mathcal{W}(Z_{\pm}) = (\xi\pm1)^{2}$.
If the real part of $\xi$ is positive (which happens exactly when the real part of $z$ is positive), we have $\mathcal{M}_2 = \{Z_+\}$.
If the real part of $\xi$ is negative (which happens when the real part of $z$ is negative), we have $\mathcal{M}_2 = \{Z_-\}$.
If $\xi$ is purely imaginary then both weights are dominant and $\mathcal{M}_2  = \mathcal{M}_1$. This is the truly oscillatory regime.

Note that
\[
R^{(4)}(Z_{\pm})\, =\, \frac{1}{4(\xi\pm 1)}\, \begin{bmatrix} 5\xi\pm 5 & \xi\pm1 & 3\xi \mp 1 & \pm 3-\xi\\
\xi\pm 1 & 5\xi\pm 5 & \pm 3 - \xi & 3\xi \mp 1 \\
  3\xi \mp 1 & \pm3-\xi& 5\xi\pm 5 & \xi\pm 1 \\
\pm 3 - \xi & 3\xi \mp 1 & \xi\pm 1 & 5\xi\pm 5
\end{bmatrix}
\]
with eigenvalues
\[
1\, ,\ 2\, ,\ \frac{2}{1\pm \xi^{-1}}\, ,\ \frac{2}{1\pm \xi}\, ,
\]
counted according to multiplicity.

Next let us consider the truly oscillatory region, and here we will also match with the known asymptotics.
If $\xi = i t$ for $t \in (0,\infty)$ then both $R^{(4)}(Z_+)$ and $R^{(4)}(Z_-)$ are strictly accretive.
Also, the weights $\mathcal{W}(Z_+)$ and $\mathcal{W}(Z_-)$ have equal magnitude, as is required to have $\mathcal{M}_2=\mathcal{M}_1$.
Then direct calculation shows that Conjecture \ref{conj:1} predicts
\[
\begin{aligned}
\frac{\operatorname{perm}(A^{(2,n)})}{(2n)!}
&=4^{-n}\left[
\frac{(1+it)^{2n+1}}{2\sqrt{it}}
+\frac{(1-it)^{2n+1}}{2\sqrt{-it}}\right]\\
&\quad+o\!\left(\left(\frac{1+t^2}{4}\right)^n\right),
\end{aligned}
\]
and this can be rewritten as
\[
\begin{aligned}
\frac{\operatorname{perm}(A^{(2,n)})}{(2n)!}
&=\left(\frac{1+t^2}{4}\right)^n
\left[\sqrt{\frac{1+t^2}{t}}\,
\cos\left((2n+1)\tan^{-1}t-\frac\pi4\right)+o(1)\right].
\end{aligned}
\]
For each fixed $t>0$, this matches the classical Legendre asymptotic with
$\theta=2\tan^{-1}t$; see \cite{NISTDLMF}, equation~18.15.4\_5 with $\alpha=\beta=0$.
The additive form is necessary because the leading cosine can vanish.

If, instead, we consider 
\[
B\, =\, \begin{bmatrix} 1 & 1 \\ 1 & \xi^2 \end{bmatrix}\, ,
\]
for $\operatorname{Re}(\xi)>0$, then 
$\mathcal{M}_2 = \{Z_+\}$.
The eigenvalues of $R^{(4)}(Z_+)$ are
\[
1\, ,\ 2\, ,\ \frac{2}{1+\xi^{-1}}\, ,\ \frac{2}{\xi+1}\, .
\]
Here $R^{(4)}(Z_+)$ is diagonalised by a real orthogonal matrix and all its eigenvalues have positive real part. Thus it is strictly accretive, so $B$ is of ``simple type.''
Then Conjecture \ref{conj:1} says
\[
\frac{1}{(2n)!}\, \operatorname{perm}\left(A^{(2,n)}\right)\, \sim\, \frac{1}{2}\, \left(\frac{\xi+1}{2}\right)^{2n}\, 
\frac{\xi+1}{\sqrt{\xi}}\, ,
\]
as $n \to \infty$.
Again, this matches the known asymptotics of the Legendre polynomial (and the central binomial coefficient).

This is essentially the unique two-dimensional example. For higher dimensional examples, the solutions of the Schr\"odinger bridge equation
are more complicated. We do know that the generic number of solutions is independent of the parameters due to Bezout's theorem, although
special singular cases may result in fewer solutions on lower dimensional subvarieties. An important complication that arises is that some
critical points that are not strictly accretive can also be relevant, if the weight of such a critical point is dominant.
This seems to require a more careful analysis such as in \cite{Pham} or possibly \cite{PemantleWilsonMelczer}.

\subsection{Additional ad hoc examples}
{\bf Example A.1: $m=3$.}
A $3\times 3$ example is 
\[
B\, =\, \begin{bmatrix} 1 & 1 & 1 \\ 
1 & 1 & -1 \\ 
1 & -1 & 1 \end{bmatrix}\, .
\]
In this case, let us use
\[
(n!)^{-3} \operatorname{perm}\left(A^{(3,n)}\right)\, =\, [x^n y^n z^n]\big((x+y+z)^n(x-y+z)^n(x+y-z)^n\big)\, .
\]
By a conjecture of Peter Bala (on OEIS), which was recently proved by Sela Fried \cite{Fried},
this is the same as 
\[
(n!)^{-3} \operatorname{perm}\left(A^{(3,n)}\right)\, =\, \sum_{k=0}^{n} \binom{n}{k}^3\, ,
\]
known as the Franel numbers.
Then by consulting
\[
\text{\url{https://oeis.org/A000172}}
\]
we can see that
\[
\sum_{k=0}^{n} \binom{n}{k}^3\, \sim\, \frac{2^{3n+1}}{\sqrt{3}\, \pi n}\, .
\]
The full set of critical points (whether stable or not) is 
$\mathcal{M}_1 = \{Z_1,Z_2,Z_3\}$ where
\[
Z_1=\begin{bmatrix}-1&1&1\\1&-1&-1\\1&-1&-1\end{bmatrix},
\qquad Z_3=\overline{Z_2},
\]
\[
Z_2=\frac12\begin{bmatrix}
1&\exp(\pi i/3)&\exp(5\pi i/3)\\
\exp(5\pi i/3)&1&\exp(4\pi i/3)\\
\exp(\pi i/3)&\exp(2\pi i/3)&1
\end{bmatrix}.
\]
Then $\mathcal{W}(Z_1)=-1$, $\mathcal{W}(Z_2)=\mathcal{W}(Z_3)=2^3$.
So $Z_2$ and $Z_3$ have dominant weights: $\mathcal{M}_2 = \{Z_2,Z_3\}$.
The eigenvalues of the $R^{(6)}(Z_2)$ are the same as the eigenvalues of $R^{(6)}(Z_3)$:
\[
1\, ,\ 2\, ,\ 1+\frac{i}{\sqrt{2}}\, ,\ 1+\frac{i}{\sqrt{2}}\, ,\ 1-\frac{i}{\sqrt{2}}\, ,\ 1-\frac{i}{\sqrt{2}}\, ,
\]
counted according to multiplicity.
These are all in $\{x+iy\, :\, x>0\, ,\, y\in\R\}$.
Moreover, the Hermitian part of each $R^{(6)}(Z_j)$, $j=2,3$, has eigenvalues
$2,1,5/4,5/4,3/4,3/4$, so both matrices are strictly accretive.
Thus this $B$ is of ``simple type.
Then Conjecture \ref{conj:1} gives
\[
\frac{1}{(3n)!}\, \operatorname{perm}\left(A^{(3,n)}\right)\, \sim\, \frac{4}{3}\, \left(\frac{2}{3}\right)^{3n}\, ,
\]
as $n \to \infty$.
This is correct, using the Bala-Fried formula.
So this $B$ satisfies Definition \ref{def:1}.

{\bf Example A.2: $m=4$.}
A $4\times 4$ example outside the hypotheses of our conjecture is 
\[
B\, =\, \begin{bmatrix} 1 & 1 & 1 & 1\\ 
1 & -1 & 1 & 1 \\ 
1 & 1 & -1 & 1 \\
1 & 1 & 1 & -1
\end{bmatrix}\, .
\]
In this case
\[
\begin{aligned}
&(n!)^{-4}\operatorname{perm}(A^{(4,n)})\\
&\quad=[x^ny^nz^nw^n]\bigl\{(x+y+z+w)^n(x-y+z+w)^n\\
&\qquad\qquad\qquad\quad{}\times(x+y-z+w)^n(x+y+z-w)^n\bigr\}.
\end{aligned}
\]
Bala has also conjectured that this is equal to $\sum_{k=0}^{n} \left(\binom{n}{k} \binom{2k}{n}\right)^2$ on 
\[
\text{\url{https://oeis.org/A290575}}
\]
On that same page, according to asymptotic analysis of Vaclav Kotesovec then this would mean
\[
(n!)^{-4} \operatorname{perm}\left(A^{(4,n)}\right)\, \sim\, (2+2\, \sqrt{2})^{2n} 2^{-3/4}  (1+\sqrt{2}) (n\pi)^{-3/2}\, ,
\]
as $n \to \infty$.
So this suggests
\[
\frac{1}{(4n)!}\, 
\operatorname{perm}\left(A^{(4,n)}\right)\, \sim\, 2^{-1/4}\left(1+\sqrt{2}\right)\, \left(\frac{3+2\sqrt{2}}{64}\right)^n\, ,
\]
as $n \to \infty$.
But one cannot recover that from our conjecture because $B$ is not of ``simple type.
We do have $\mathcal{M}_1=\{Z_+,Z_-\}$ for
\[
Z_{\pm}\, =\, \begin{bmatrix} -2 & 1 & 1 & 1 \\ 
1 & 0 & 0 & 0 \\
1 & 0 & 0 & 0 \\
1 & 0 & 0 & 0 \end{bmatrix}
\pm \frac{1}{\sqrt{2}}
\begin{bmatrix}
3 & -1 & -1 & -1\\
-1 & 1 & 1 & 1\\
-1 & 1 & 1 & 1\\
-1 & 1 & 1 & 1
\end{bmatrix}\, .
\]
Obviously $Z_+$ has a dominant weight, and $Z_-$ has a subdominant weight.
But for neither of these matrices is the resulting $R^{(8)}(Z)$ matrix strictly accretive.
So $B$ is not of ``simple type.''
It is true that $\mathcal{W}(Z_+)^n/4^{4n}=\left(\frac{3+2\sqrt{2}}{64}\right)^n$.
But the eigenvalues of $R^{(8)}(Z_+)$ are
\[
1+\sqrt{2}\, ,\ 1+\sqrt{2}\, ,\ 2(2-\sqrt{2})\, ,\ 1\, ,\ 2\, ,\ 1-\sqrt{2}\, ,\ 1-\sqrt{2}\, ,\ 2(\sqrt{2}-1)
\]
counted according to multiplicity.
(So, for example, $\sqrt2/\sqrt{\det(R^{(8)}(Z_+))}$ is  not $2^{-1/4} (1+\sqrt{2})$.)
Further analysis would need to be done for this type of example, presumably based on the reference \cite{Pham}
and/or \cite{PemantleWilsonMelczer}.

\section{Permanent of random matrices}
Aaronson has guessed more than the Aaronson-Arkhipov conjecture. He also guessed that the distribution of the logarithms of the absolute values
of the permanents are normally distributed.
One may expect that if one were to run a Ginibre process
\[
d(A(t))_{j,k}\, =\, dB^{(1)}_{j,k}(t) + i dB^{(2)}_{j,k}(t)\, ,
\]
for $B^{(1)}_{j,k}(t)$, $B^{(2)}_{j,k}(t)$ IID standard (real-valued) Brownian motions,
then the quantity
\[
\ln(|\operatorname{perm}(A(t))|)
\]
would be well approximated by a real Brownian motion.
In fact this seems to be the case numerically according to a recent investigation by Rivin \cite{Rivin}.
But he points out that this surmise needs to be completely revised if one replaces the complex Ginibre
ensemble by the real Ginibre ensemble. Deriving an exact stochastic differential equation for the permanent of a finite $N \times N$ matrix is difficult due to the lack of unitary invariance. Instead, our approach here is to bypass the microscopic combinatorial complexity by taking the macroscopic block-uniform limit ($n \to \infty$) first. We analyze the continuous-time evolution of the optimal transport plan (the Schrödinger bridge) under these stochastic matrix perturbations. By utilising linear response theory and Itô calculus, we can track the stochastic drift of the large deviation rate function $\Lambda$, providing a local macroscopic calculation of the differing drift behaviours under real versus complex noise.
\subsection{Linear response theory}

Consider a differentiable 1-parameter family of matrices $B(t)$, and let us consider how
the solution of the Schr\"odinger bridge equation evolves.
Writing $\Delta B(t) = B(t+\Delta t) - B(t)$ and similarly for $\Delta \boldsymbol{x}(t)$
and $\Delta \boldsymbol{y}(t)$,
we have
\begin{equation}
\label{eq:diff1}
\forall r \in [m]\, ,\ 
\left(x_r + \Delta x_r\right) \sum_{s=1}^m \left(b_{r,s} + \Delta b_{r,s}\right)\left(y_s + \Delta y_s\right)\, =\, 1\, ,
\end{equation}
where we suppress the $t$ dependence,
and
\begin{equation}
\label{eq:diff2}
\forall s \in [m]\, ,\ 
\sum_{r=1}^{m} \left(x_r + \Delta x_r\right) \left(b_{r,s} + \Delta b_{r,s}\right)\left(y_s + \Delta y_s\right)\, =\, 1\, .
\end{equation}
If we keep terms only linear in $\Delta$'s, then we have
\[
\begin{bmatrix} I & T\\ T^T & I \end{bmatrix}
\begin{bmatrix} X^{-1} & 0 \\ 0 & Y^{-1} \end{bmatrix} \begin{bmatrix} \Delta \boldsymbol{x} \\ \Delta \boldsymbol{y}
\end{bmatrix}\,
=\, -\begin{bmatrix} m^{1/2} X (\Delta B) Y \boldsymbol{\chi} \\ 
m^{1/2} \big(X(\Delta B)Y\big)^T \boldsymbol{\chi} \end{bmatrix}\, ,
\]
where $X=\operatorname{diag}(x_1,\ldots,x_m)$, $Y=\operatorname{diag}(y_1,\ldots,y_m)$, $T = XBY$ and $\boldsymbol\chi=\boldsymbol\chi^{(m)}$. We use that $T \boldsymbol{\chi} = \boldsymbol{\chi}$ and $T^T \boldsymbol{\chi} = \boldsymbol{\chi}$.
As usual there is a natural zero-mode for  the left-most  matrix on the left-hand-side,  $[\boldsymbol{\chi}^T,-\boldsymbol{\chi}^T]^T$,
but the right-hand-side is orthogonal to this vector. Let
\[
\left(\begin{bmatrix} I & T\\ T^T & I \end{bmatrix}^{\perp}\right)^{-1}
\]
indicate the inverse of the matrix projected to the orthogonal complement of the range of
\[
\begin{bmatrix} J/2 & -J/2 \\ -J/2 & J/2 \end{bmatrix}\, .
\]
So in linear response theory, we have
\[
\begin{bmatrix} \Delta \boldsymbol{x} \\ \Delta \boldsymbol{y}\end{bmatrix}\,
=\, - \begin{bmatrix} X & 0 \\ 0 & Y \end{bmatrix}
\left(\begin{bmatrix} I & T\\ T^T & I \end{bmatrix}^{\perp}\right)^{-1}
\begin{bmatrix} m^{1/2} X (\Delta B) Y \boldsymbol{\chi} \\ 
m^{1/2} \big(X(\Delta B)Y\big)^T \boldsymbol{\chi} \end{bmatrix}\, .
\]

\begin{remark}
If $B$ were strictly positive then we would know that the largest singular value of $T$ would be $\sigma_1=1$,
and it would be nondegenerate. (Its multiplicity would be 1.)
For real matrices, the reduced linear-response operator becomes singular if a second singular value of $T$ reaches $1$.
For complex matrices, the corresponding condition is that $TT^T$ has an eigenvalue $1$ on $\boldsymbol\chi^\perp$; here $T^T$ is the transpose, not the adjoint.
\end{remark}

\subsection{It\^o  correction: drift}

We now compute the stochastic drift of the large deviation rate function $\Lambda$ under continuous Ginibre noise. This gives a local calculation illustrating the real-versus-complex dichotomy discussed by Rivin \cite{Rivin}: real noise induces a negative drift, while complex noise yields a local martingale. The real calculation is stopped before $B(t)$ leaves the strictly positive cone. In the complex calculation, we choose a locally holomorphic critical-point branch and logarithms, and stop before the reduced Hessian becomes singular or that branch ceases to be defined.

Define the function
\begin{equation}
f_B(\boldsymbol{x},\boldsymbol{y})\, =\, \sum_{r=1}^{m} \sum_{s=1}^{m} x_r b_{r,s} y_s - \sum_{r=1}^{m} \ln(x_r) - \sum_{s=1}^{m} \ln(y_s)\, .
\end{equation}
Recall from our steepest descents analysis that the optimal scaling vectors $\boldsymbol{x}^*$ and $\boldsymbol{y}^*$ are the critical points of $f_B$. At this optimum, the transport matrix $T$ is doubly stochastic, so the sum of its entries is $m$. Evaluating the function at the critical point yields
\begin{equation}
f_B(\boldsymbol{x}^*,\boldsymbol{y}^*)\, =\, m - \sum_{r=1}^{m} \ln(x_r^*) - \sum_{s=1}^{m} \ln(y_s^*)\, .
\end{equation}
Consequently, the rate function $\Lambda$ evaluated at $B$ is given by
\begin{equation}
\Lambda(B)\, =\, f_B(\boldsymbol{x}^*(B),\boldsymbol{y}^*(B)) - m\, .
\end{equation}
Here $\Lambda(B)$ uses the unnormalised block convention. For the corresponding positive block cost $\mathcal C$, the convention of the introduction is
\[
\Lambda[\mathcal C]=-\ln m+\frac1m\Lambda(B).
\]
Thus $-m\ln m+\Lambda(B)$ is the exponent per block size $n$ in Theorem~\ref{thm:main}.

Consider $B(0) = J^{(m)}$, and let $dB(t)$ be an $m\times m$ matrix whose entries are independent differentials of standard Brownian motion. By the multivariate chain rule, the first variation of $\Lambda$ with respect to a matrix entry $b_{r,s}$ is
\begin{equation}
\frac{\partial \Lambda}{\partial b_{r,s}}\, =\, \frac{\partial f_B}{\partial b_{r,s}} + \sum_{k=1}^{m} \frac{\partial f_B}{\partial x_k} \frac{\partial x_k}{\partial b_{r,s}} + \sum_{k=1}^{m} \frac{\partial f_B}{\partial y_k} \frac{\partial y_k}{\partial b_{r,s}}\, .
\end{equation}
Since $(\boldsymbol{x}^*, \boldsymbol{y}^*)$ is a critical point of $f_B$, the partial derivatives with respect to $x_k$ and $y_k$ vanish. By the envelope theorem, the derivative reduces to the direct partial derivative:
\begin{equation}
\label{eq:envelope1}
\frac{\partial \Lambda}{\partial b_{r,s}}\, =\, x_r y_s\, .
\end{equation}

Applying It\^o's formula to $\Lambda(B(t))$ gives
\begin{equation}
d\Lambda\, =\, \sum_{r=1}^{m} \sum_{s=1}^{m} \left(\frac{\partial \Lambda}{\partial b_{r,s}}\right) db_{r,s} + \frac{1}{2} \sum_{r=1}^{m} \sum_{s=1}^{m} \left(\frac{\partial^2 \Lambda}{\partial b_{r,s}^2}\right) (db_{r,s})^2\, .
\end{equation}
Substituting (\ref{eq:envelope1}) into the expansion and using the quadratic variation $(db_{r,s})^2 = dt$ for the real Ginibre ensemble, we obtain
\begin{equation}
d\Lambda\, =\, \sum_{r=1}^{m} \sum_{s=1}^{m} x_r y_s db_{r,s} + \frac{1}{2} \sum_{r=1}^{m} \sum_{s=1}^{m} \frac{\partial (x_r y_s)}{\partial b_{r,s}}\, dt\, .
\end{equation}

By the product rule, $\partial(x_r y_s) / \partial b_{r,s} = y_s (\partial x_r / \partial b_{r,s}) + x_r (\partial y_s / \partial b_{r,s})$. The linear responses $\partial \boldsymbol{x} / \partial B$ and $\partial \boldsymbol{y} / \partial B$ were computed in the preceding section in terms of the pseudo-inverse $\mathcal{G} = \left(\begin{bmatrix} I & T\\ T^T & I \end{bmatrix}^{\perp}\right)^{-1}$. Substituting these terms yields the drift:
\begin{equation}
d\Lambda_{\text{drift}}\, =\, - \frac{1}{2}\, \sum_{r=1}^{m} \sum_{s=1}^{m} x_{r}^2 y_{s}^2
\begin{bmatrix} \boldsymbol{e}_{r} \\ \boldsymbol{e}_{s} \end{bmatrix}^T \mathcal{G}
\begin{bmatrix} \boldsymbol{e}_{r} \\ \boldsymbol{e}_{s} \end{bmatrix}\, dt\, .
\end{equation}
Since $\mathcal{G}$ is positive-definite on the subspace orthogonal to the zero-mode, this quadratic form is positive, yielding a strictly negative drift for $\Lambda$ under real Ginibre noise.

This formulation also clarifies the complex Ginibre ensemble case. If the noise is complex, the differential is $dB = dR + i dM$ for independent real Brownian motions $dR$ and $dM$. The quadratic variation is then
\begin{equation}
(db_{r,s})^2\, =\, (dr_{r,s} + i dm_{r,s})^2\, =\, (dr_{r,s})^2 - (dm_{r,s})^2\, =\, dt - dt\, =\, 0\, .
\end{equation}
The real and imaginary quadratic variations cancel out. Thus, the It\^o drift vanishes, and $\Lambda$ evolves as a local martingale with no macroscopic drift.

\section*{Acknowledgements}
We are grateful to Sumit Mukherjee, Soumik Pal and Raghavendra Tripathi for useful conversations.


\begin{thebibliography}{99}

\bibitem{AaronsonArkhipov}
Aaronson, S. and Arkhipov, A. (2011) The computational complexity of linear optics.
In \emph{Proceedings of the Forty-Third Annual ACM Symposium on Theory of Computing (STOC '11)}, pp.~333--342.

\bibitem{Abdesselam}
Abdesselam, A. (2026) A central limit theorem for connected components of random coverings of manifolds with nilpotent fundamental groups.
Preprint, \url{https://arxiv.org/abs/2603.24499}.

\bibitem{Cuturi2013}
Cuturi, M. (2013) Sinkhorn distances: lightspeed computation of optimal transportation distances.
In \emph{Proceedings of the 27th International Conference on Neural Information Processing Systems}, Vol.~2, pp.~2292--2300.

\bibitem{FlajoletSedgewick}
Flajolet, P. and Sedgewick, R. (2009) \emph{Analytic Combinatorics}. Cambridge University Press, Cambridge.

\bibitem{Fried}
Fried, S. (2026) Proof of ten conjectures from the OEIS.
\emph{J. Integer Seq.} \textbf{29}, Article 26.1.8.

\bibitem{Galichon}
Galichon, A. (2016) \emph{Optimal Transport Methods in Economics}. Princeton University Press, Princeton, NJ.

\bibitem{HLP}
Harchaoui, Z., Liu, L. and Pal, S. (2024) Asymptotics of discrete Schr\"odinger bridges via chaos decomposition.
\emph{Bernoulli} \textbf{30}(3), 1945--1970.

\bibitem{HossainStarr}
Hossain, S. and Starr, S. (2024) About the moments of the generalized Ulam problem.
Preprint, \url{https://arxiv.org/abs/2404.05860}.

\bibitem{Huh}
Huh, J. (2025) Quantum estimation bound of the Gaussian matrix permanent.
\emph{Phys. Rev. A} \textbf{111}, Article 012418.

\bibitem{ItzyksonDrouffe}
Itzykson, C. and Drouffe, J.-M. (1989) \emph{Statistical Field Theory, Vol.~1: From Brownian Motion to Renormalization and Lattice Gauge Theory}. Cambridge University Press, Cambridge.

\bibitem{Leonard}
L\'eonard, C. (2012) From the Schr\"odinger problem to the Monge--Kantorovich problem.
\emph{J. Funct. Anal.} \textbf{262}(4), 1879--1920.

\bibitem{McCullagh}
McCullagh, P. (2014) An asymptotic approximation for the permanent of a doubly stochastic matrix.
\emph{J. Stat. Comput. Simul.} \textbf{84}(2), 404--414.

\bibitem{Muirhead}
Muirhead, R.~J. (1982) \emph{Aspects of Multivariate Statistical Theory}. John Wiley and Sons, Hoboken, NJ.

\bibitem{Mukherjee}
Mukherjee, S. (2016) Estimation in exponential families on permutations.
\emph{Ann. Statist.} \textbf{44}(2), 853--875.

\bibitem{NISTDLMF}
National Institute of Standards and Technology. \emph{Digital Library of Mathematical Functions},
Section~18.15, equation~18.15.4\_5.
\url{https://dlmf.nist.gov/18.15.E4_5}.

\bibitem{Pal}
Pal, S. (2024) Limiting partition function for the Mallows model: a conjecture and partial evidence.
Preprint, \url{https://arxiv.org/abs/2406.18855}.

\bibitem{PemantleWilsonMelczer}
Pemantle, R., Wilson, M.~C. and Melczer, S. (2024) \emph{Analytic Combinatorics in Several Variables}, 2nd edn. Cambridge University Press, Cambridge.

\bibitem{Pham}
Pham, F. (1967) \emph{Introduction \`a l'\'etude topologique des singularit\'es de Landau}. Gauthier-Villars, Paris.

\bibitem{Pinsky1}
Pinsky, R.~G. (2006) Law of large numbers for increasing subsequences of random permutations.
\emph{Random Structures Algorithms} \textbf{29}(3), 277--295.

\bibitem{Pinsky2}
Pinsky, R.~G. (2007) When the law of large numbers fails for increasing subsequences of random permutations.
\emph{Ann. Probab.} \textbf{35}(2), 758--772.

\bibitem{Rivin}
Rivin, I. (2026) Permanents of matrix ensembles: computations, distribution and geometry.
Preprint, \url{https://arxiv.org/abs/2602.10141}.

\bibitem{Simon}
Simon, B. (1993) \emph{The Statistical Mechanics of Lattice Gases, Vol.~I}. Princeton University Press, Princeton, NJ.

\bibitem{Sinkhorn1964}
Sinkhorn, R. (1964) A relationship between arbitrary positive matrices and doubly stochastic matrices.
\emph{Ann. Math. Statist.} \textbf{35}(2), 876--879.

\bibitem{Tripathi}
Tripathi, R. (2026) On the partition function of a class of Mallows model.
Preprint, \url{https://arxiv.org/abs/2605.03647v1}.

\bibitem{Valiant}
Valiant, L.~G. (1979) The complexity of computing the permanent.
\emph{Theoret. Comput. Sci.} \textbf{8}(2), 189--201.

\end{thebibliography}
\end{document}